\documentclass[11pt]{article} 
\usepackage[utf8]{inputenc}

\usepackage{amsmath}
\usepackage{amsthm}
\usepackage{amsfonts} 
\usepackage{color}
\usepackage{geometry} 
\usepackage{setspace} 
\usepackage{bbold}
\usepackage{booktabs} 
\usepackage{array} 
\usepackage{paralist} 
\usepackage{verbatim} 
\usepackage{subfig}
\usepackage{sectsty}
\usepackage[nottoc,notlof,notlot]{tocbibind} 
\usepackage[titles,subfigure]{tocloft} 
\allsectionsfont{\sffamily\mdseries\upshape}
\usepackage{bm} 
\newtheorem{theorem}{Theorem}
\newtheorem{lemma}{Lemma}
\newtheorem{corollary}{Corollary}
\newtheorem{proposition}{Proposition}

\newtheorem{definition}{Definition}

\newtheorem{remark}{Remark}
\newcommand{\beq}{\begin{equation}}
\newcommand{\eeq}{\end{equation}}

\newcommand{\Z}{\mathbb{Z}}
\newcommand{\N}{\mathbb{N}}
\newcommand{\T}{\mathbb{T}}
\newcommand{\R}{\mathbb{R}}

\renewcommand{\a}{\alpha}

\newcommand{\ba}{\begin{array}}
\newcommand{\ea}{\end{array}}

\newcommand{\bal}{\begin{aligned}}
\newcommand{\eal}{\end{aligned}}

\newtheorem{claim}{Claim}

\UseRawInputEncoding

\usepackage{graphicx,psfrag,epsfig,multirow}
\def\cW{\mathcal W}
\def\cG{\mathcal G}

\def\cE{\mathcal E}

\def\cP{\mathcal P}
\def\eps{\epsilon} 
\def\Nzero{\mathcal N_0}
\def\None{\mathcal N_1}
 
\def\vphi{\varphi} 
\def\jjj{\{0,\ldots,4\}}
\makeatletter 
\@addtoreset{equation}{section}
\makeatother

\title{Analytic reparametrizations of translation toral flows with countable Lebesgue spectrum}
\author{Fatna Abdedou, Bassam Fayad, Arezki Kessi}
\date{}

\begin{document}
\maketitle
\begin{abstract}
   We give an example of a  real analytic reparametrization of a minimal translation flow on $\mathbb{T}^{5}$ that has a Lebesgue spectrum with infinite multiplicity.
\end{abstract}

\section{Introduction} Inspired by Kolmogorov's 1954 ICM talk \cite{Kicm}, the following questions arise: {\it Can a completely integrable Hamiltonian flow be perturbed in the real analytic category so that the perturbed flow has  an invariant torus exhibiting dynamics with a maximal spectral type equivalent to the Lebesgue measure on $\R$? Can a reparametrized translation flow of the torus have a maximal spectral type that is equivalent to the Lebesgue measure on $\R$? }


By a classical construction that we recall below, a positive answer to the second question, with reparametrization functions arbitrarily close to 1,  immediately yields examples of perturbations that answer positively the first question.  The aim of this paper is to prove the following. Let $\T=\R/\Z$.

\begin{theorem}\label{Theo}
There exists a minimal translation flow on $\T^5$, and a strictly positive real entire function $\Phi$ defined on $\mathbb{T}^{5}$, such that the reparametrization of the irrational flow $(\alpha,1)$ by $\Phi$ has a Lebesgue spectrum with infinite multiplicity. Moreover, $\Phi$ can be chosen arbitrarily close to $1$ on any bounded domain.  
 \end{theorem}

Note that arbitrarily small perturbations of the reparametrized flow, such as changing slightly $\a$ or  $\Phi$, can render the flow analytically conjugated to a translation flow, who has a pure point spectrum.

 Since the reparametrization function can be chosen arbitrarily close to $1$, it is a classical observation that this automatically gives examples of perturbations of integrable Hamiltonians with invariant tori carrying uniquely ergodic dynamics with infinite Lebesgue spectrum (ILS). Moreover, the integrable Hamiltonian can be chosen to be convex. Indeed,  consider the completely integrable system on $\T^5\times \R^5$ given by
$$H(\theta,r)=\Phi(\theta)(\frac{1}{2}\sum_{j=1}^5 r_j^2-1).$$
Denote $X^t_H$ the corresponding flow. The energy surface $\{H=0\}$ is foliated by invariant tori for $X^t_H$ on which the restricted flow is a reparametrization by $\Phi$ of the translation flow of frequency vector $r$. 
One can then pick the vector $r$ and the function $\Phi$ to be as in the theorem and get the invariant torus with ILS. 

Kolmogorov~\cite{kolmogorov} showed that  reparametrizations of translation flows on $\T^2$ are typically conjugated to translation flows, since it suffices for this that the slope of the translation flow belongs to the full measure set of  Diophantine numbers.  He also observed that more exotic behaviors should be expected for the reparametrized flows in the case of Liouville slopes. Shklover indeed obtained in \cite{shklover} examples of real analytic reparametrizations of linear flows on the $2$-torus that were weak mixing (continuous spectrum). 
Not long after Shklover's result, Katok \cite{katok} and  Kochergin~\cite{Koc1} showed the absence of mixing for non-singular conservative $C^1$ flows on the $2$-torus. Such flows can be viewed as special flows above a circular rotation $R_\a$ and the absence of mixing is based on the Denjoy-Koksma cancellation property (DKP) for Birkhoff sums above an irrational rotation. In the case of a special flow with a $C^1$ ceiling function of average $1$, one can easily see from the DKP that the special flow satisfies $T^{q_n} \to {\rm Id}$ uniformly, where $q_n$ is the sequence of denominators of the best rational approximations (or convergents) of $\a$.

For reparametrizations of minimal translation flows on higher dimensional tori the situation is quite different. Kolmogorov's linearization result still holds for almost all frequency vectors as observed by Herman. However, Yoccoz showed in \cite{Y} that the Denjoy-Koksma cancellation property has no counterpart in higher dimensions. Using frequency vectors $\a \in \R^2$ with coordinates that are Liouville numbers with a special alternating configuration of the denominators of their convergents, he constructed a function of which the Birkhoff sums above the (minimal) translation $R_\a$ do not have the DKP.  

More precisely,  
the main ingredient in the construction of \cite{Y} is the use of a vector $\a=(\a_1,\a_2)$ such that the denominators,
$q_n^{(1)}$ and $ q_n^{(2)}$
of the convergents  of $\a_1$ and $\a_2$ are alternated and such that each term in the increasing sequence $\ldots,q_n^{(1)},q_n^{(2)},q_{n+1}^{(1)},q_{n+1}^{(2)},\ldots$ is exponentially larger than the precedent one. 
Then  \cite{Y} constructs a real analytic function   
$A:\T^2 \to \mathbb C$ 
with zero integral such that for almost every $(x_1,x_2) \in \T^2$,  $|A_m(x_1,x_2)| \to \infty$ as $m\to \infty$.

Using the construction of Yoccoz, the second author constructed in \cite{Fanalytic} mixing analytic reparametrizations for a class of minimal translation flows on $\T^3$. These flows can be viewed as special flows above a minimal translations $R_\a$ of the two torus with $\a$ as in \cite{Y}, and the mixing mechanism comes from the uniform stretch of the Birkhoff sums of an adequately chosen ceiling function $\vphi$ above $R_\a$ (see Figure Section \ref{sec.def} and Section \ref{sec.def} for the definition of the uniform stretch measurement $S_J^t$).

 To be more precise,  the
disposition of the best approximations of $\a_1$ and $\a_2$, allows to consider a function $\vphi$ that is a sum of two functions $\vphi_1(x_1)+\vphi_2(x_2)$ such that the  ergodic
sums $\varphi_m$ of the function $\varphi$, for any  $m$ sufficiently
large, will be always stretching (i.e. have big derivatives at most points), in one
or in the other of the two directions, $x_1$ or $x_2$, depending on
whether $m$ is far from $\{q_n\}$ or far from $\{q'_n\}$. And this stretch
will increase when $m$ goes to infinity. So when time goes from $0$ to
$t$, $t$ large, the image of a small typical interval $J$ from the basis
${\T}^2$ (depending on $t$ the intervals should be taken along the $x_1$
or the $x_2$ axis) will be more and more  stretched  in the
fibers' direction, until the image of $J$ at time $t$ will consist of
a lot of almost vertical curves whose projection on the basis lies
along a piece of a trajectory under the translation $R_{\a}$. By
unique ergodicity these projections become more and more uniformly
distributed, and so will $T_{\a,\vphi}^t(J)$. For each $t$, and except for
increasingly small subsets of it (as function of $t$), it is possible 
to cover the basis with such ''good'' intervals. Besides, what is
true for $J$ on the basis is true for $T_{\a,\vphi}^s(J)$ at any bounded height $s$ on the
fibers. So applying Fubini Theorem, we will obtain the asymptotic uniform
distribution of any measurable subset, which is, by definition, the
mixing property (see Figure \ref{FigFlowMix}).

\begin{figure}[htb]
    \centering
    \resizebox{!}{3cm}{\includegraphics{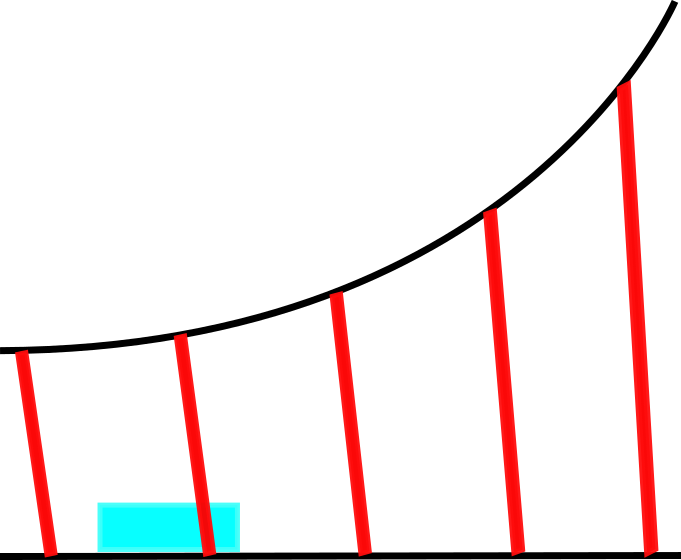}}

\caption{Mixing mechanism for special flows: the image of a rectangle is
a union of long narrow strips which fill densely the phase space.} 
    \label{FigFlowMix}
\end{figure}

Uniform stretch is also responsible for mixing in the conservative surface flows  with one degenerate singularity studied by Kochergin in the 1970s \cite{Koc2}.  Kochergin flows are special 
flows under an integrable ceiling  function with at least one power singularity (see Figures \ref{sym} and \ref{orbits}  and the precise definition of special flows in Section \ref{sec.def}). The uniform stretch of the Birkhoff sums in this context comes from the shear between different orbits as they go near the singularity. 

\begin{figure}[htb]
 \centering
  \resizebox{!}{4cm}{\includegraphics[angle=0]{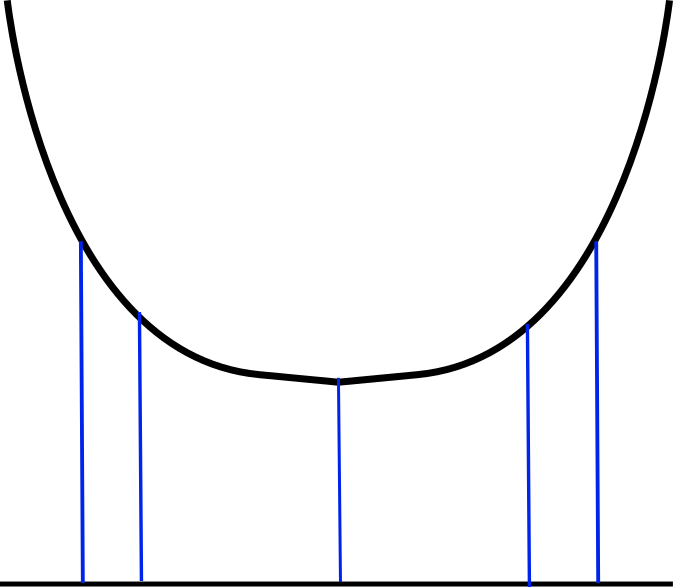}}
\caption{\small  Representation of a $2$-torus flow with one degenerate saddle  as a special flow under a ceiling function with a  power-like singularity.} 
  \label{sym}
\end{figure}

The paper \cite{FFK} proved ILS for Kochergin flows. To start with, it established square summable decay of correlations for observables that are smooth coboundaries above the flow. This suffices to conclude that the maximal spectral type is absolutely continuous with respect to the Lebesgue measure. The use of  smooth coboundaries to get faster decay estimates than for general smooth observables or characteristic functions was inspired from \cite{FU} that established Lebesgue maximal spectral type for time changes of horocyclic flows. The crucial point in the use of  smooth coboundaries is that it allows to translate all the uniform stretch in speed of mixing. 

To establish ILS for Kochergin flows, as well as for time changes of horocyclic flows, the paper \cite{FFK}  introduced a criterion based on the decay of correlations that allows to deduce the ILS property. The criterion exploits the speed of equi-distribution of small sets  transversal to the flow direction to build observables that have almost orthogonal cyclic spaces and that have a spectral type close to being Lebesgue. An abstract statement then allows to conclude the ILS property. 

The delicate point with Kochergin flows, in comparison with time changes of horocyclic flows for example, is that the shear between orbits is not uniform. Equivalently, the stretch of the Birkhoff sums of the ceiling function can vanish or be very weak on parts of the phase space in a way that is also not uniform in time. The key factor is how close to the singularity different orbits get to before time $t$.  

In the mixing reparametrizations of translation flows the same phenomenon of non uniform shear in space and time appears because the DKP still appears for some special times on part of the phase space. A precise computation shows that, at time $t$, given a direction of uniform stretch, the set with uniform stretch weaker than $t^{1-2\beta}$ is of measure that can be compared to $t^{-\beta+\epsilon}$, for any $\epsilon>0$ (see Corollary \ref{coro.good}). This means that for times with only one direction of stretch as it happens for the special flows above $\T^2$, 
we will get a uniform stretch larger than $t^{1/2}$ away from a set of measure $t^{-1/4+\epsilon}$. The size of the bad set in this case is too large since we seek a square summable decay of correlations. However, if for each $t$ there are 3 independent directions of stretch, the stretch will be stronger than $t^{1/2+\epsilon}$ on a set of measure $t^{-3/4+3\epsilon}=o(t^{-1/2-\epsilon})$. 

A simple observation concerning the construction of the frequency vectors as in \cite{Y,Fanalytic} allows to construct frequencies  $\a\in \R^4$ and ceiling functions $\vphi$ such that  for each time $t$ there are three directions of uniform stretch and gives examples of special flows above $\T^4$ translations with ILS. This corresponds to reparametrizations of minimal flows on $\T^5$. Our method, that insures that for each $t$ the decay is less than $t^{-1/2-\eps}$, does not allow to treat the case of $\T^3$ and $\T^4$. 

In principle one could apply the strategy adopted in \cite{FFK} to the case of the mixing reparametrizations on $\T^3$ or $\T^4$. That is, 
accept the existence of a sequence of special times $(t_n)$ (such as the multiples of the denominators of the convergents) at which the small measure sets that are bad (with no strong uniform stretch) lead to a decay slower than $t_n^{-1/2}$, but still recover square summable correlations due to 
the fact that for most of the times that are in a medium scale neighborhood of the times $t_n$, there is some small power decay of correlations on the bad set itself. This is what was done for Kochergin flows in \cite{FFK} because the decay was slower than $t^{-1/2}$ for some special times. We believe the same can be done for  reparametrized mixing flows on $\T^4$ and possibly on $\T^3$, but the proof would then be much more technical than the one given here for reparametrizations on $\T^5$.

\section{Notations and definitions} \label{sec.def}


\begin{itemize}
\item {\it Special flows above translations of the torus.} Let $R_\a:\T^d\to\T^d$, $R_\a(\theta)=\theta+\a \; \text{\rm mod }1$, where $\a\in [0,1]^d$ is a vector such that $1,\a_1,\ldots,\a_d$ are independent over $\Z$.  
 Let $\vphi \in L^1(\T^d,\lambda_{\T^d})$ be a strictly positive function. 
 We recall that the special flow $T^t:=T^t_{\a,{\vphi}}$ constructed above $R_\a$ and under $\vphi$ is the flow defined almost everywhere  by 
\begin{eqnarray*}
\T^d \times \R / \sim  &  \rightarrow &  \T^d \times \R / \sim  \\
 (\theta,s) & \rightarrow & (\theta,s+t), \end{eqnarray*}
where $\sim$ is the identification, defined on $\T^d \times \R$, 
\begin{equation*}
(\theta, s + \vphi(\theta)) \sim (R_\a(\theta),s) \,.
\end{equation*}
Equivalently (see Figure \ref{orbits}), this special flow  is defined for $s\geq 0$ and for all times $t\in \R$ such that $t+s \geq 0$ (with a 
similar definition for times $t\in \R$ such that $t+s < 0$) by 
\begin{equation}\label{eq:refspec}T^t(\theta,s) = (\theta+N(\theta,s,t)\a , t+s-   S_{N(\theta,s,t)}\vphi (\theta)),
\end{equation} 
   where $N(\theta,s,t)$ is the unique integer such that 
\begin{equation}
\label{D-C}
0 \leq  t+s-   S_{N(\theta,s,t)}{\vphi} (\theta) \leq {\vphi}(\theta+N(\theta,s,t)\a),\end{equation}    
and
$$
S_n{\vphi}(\theta)=\left\{\begin{array}{ccc}
{\vphi}(\theta)+\ldots+{\vphi}(R_\alpha^{n-1}\theta) &\mbox{if} & n>0\\
0&\mbox{if}& n=0\\
-({\vphi}(R_\alpha^n\theta)+\ldots+{\vphi}(R_\alpha^{-1}\theta))&\mbox{if} &n<0.\end{array}\right.$$

\begin{figure}[htb] 
 \centering
  \resizebox{!}{4cm}{\includegraphics[angle=0]{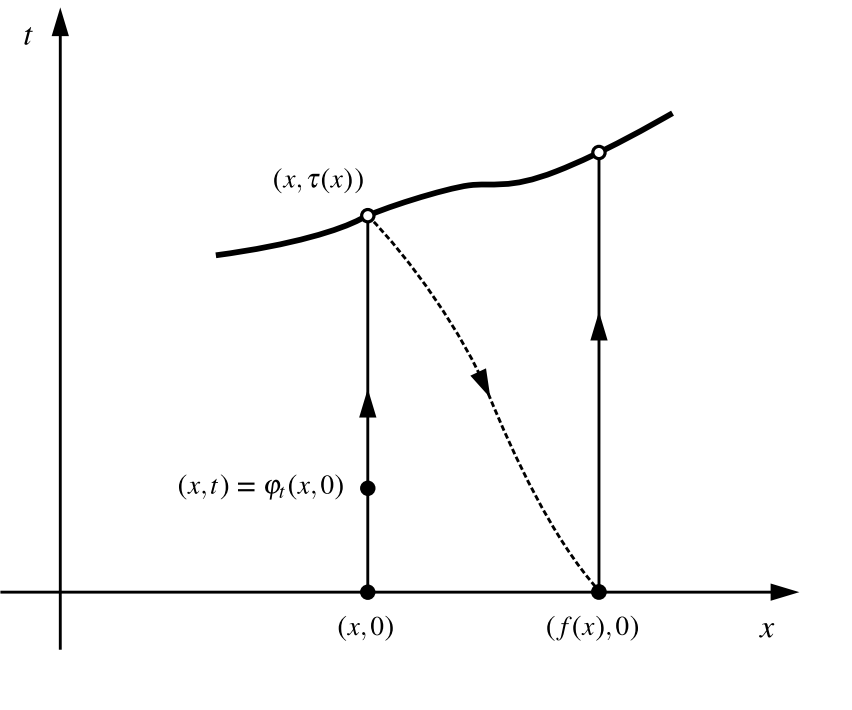}}  
\caption{\small The orbit of a point by the special flow above a transformation $f$ and under a bounded ceiling function $\tau$.}
\label{orbits}
\end{figure}

\item We recall the notations $M=\{(x,s)\in \T^4\times\R\;:\; 0\leq s< \vphi(x)\}$ for the configuration space of the flow $T^t_{\a,\vphi}$ and  $\mu$ for the measure  equal to the restriction to $M$ of the product of the Haar measures $\lambda_4:=\lambda_{\T^4}$ on the torus $\T^4$ and $\lambda:=\lambda_\R$ on the real line~$\R$. 

For a given $\zeta>0$, let us denote
\begin{equation}
\label{Mzeta}
M_\zeta:=\{(x,s)\in M\;:\;  0\leq s \leq \vphi(x)-\zeta\}.
\end{equation}


\item For $j\in \jjj$, we denote by $I^j$ intervals in the $j$-direction : $I^j\subset M$ is of the form $\{z=(x_1,x_2,x_3,x_4,s) \in M : x_j \in [a,b]\}$, for some $0\leq a<b\leq 1$. We also call such intervals $j$-intervals. 

\item For a point $z\in M$, we denote by $\bar z$ its projection on the base $\T^4$. We denote by $\pi_j(z)$ the projection of $z$ on the $j^{\rm th}$-coordinate of $\T^4$.

\item On a $j$-interval $I^j$, we denote by $\lambda$ the Lebesgue measure.

\item Given an interval $I \subset M$ (in any direction) and $t\in \R$, we define  
\begin{align} \label{def.rt} 
r^{t}_{I}&=\inf_{z\in I}\lvert \varphi'_{N(z,t)}(\bar z)\rvert \\
   \label{def.st}  S^{t}_{I}&=\inf_{x\in I}\frac{(\varphi'_{N(z,t)}(\bar z))^{2}}{\varphi''_{N(z,t)}(\bar z)}.
\end{align}

The quantity $S^{t}_{I}$ measures the {\it uniform stretch} of the Birkhoff sums above the interval $I$. Indeed, when $S^t_I$ is large, this corresponds to almost linear expansion of the interval $I\times \{0\} \subset M$ under the flow at time $t$.  
\end{itemize}

\section{The construction and precise statements}

\subsection{The choice of the frequency vector}
Following \cite{Y} (see also \cite{Fanalytic} in the context of special flows), let $Y\subset [0,1]^4$ be the set of vectors $\alpha:=(\alpha_{1},...,\alpha_{4})$ whose sequences of best approximations $q_{n}^{(1)}$,..., $q_{n}^{(4)}$ satisfy the following for all $n\geq 1$ :
 \begin{equation}\label{Arith1}
q_{n}^{(2)}\geq e^{nq_{n}^{(1)}}, \quad q_{n}^{(3)}\geq e^{nq_{n}^{(2)}},  \quad q_{n}^{(4)}\geq e^{nq_{n}^{(3)}},  \quad q_{n+1}^{(1)}\geq e^{nq_{n}^{(4)}}.
\end{equation}

As in \cite{Y}, we have that 
\begin{lemma}\label{Arithmetric}
The set $Y$ is a dense continuum in $[0,1]^4$. 
\end{lemma}
\begin{proof}[Proof of Lemma \ref{Arithmetric}]
First recall that any irrational number $\alpha \in \mathbb{R-Q}$ can be
written as a continued fraction expansion 
\begin{equation*}
\a=\left[ a_{0},a_{1},...\right]
=a_{0}+1/(a_{1}+1/(a_{2}+...1/a_{n-1}+1/a_{n})...),
\end{equation*}%
where $\left\{ a_{j}\right\} _{j\geq 1}$ is a sequence of integers $\geq 1$
and $a_{0}=\left[ \alpha \right] .$
Conversely, any infinite sequence $\left\{ a_{j}\right\} _{j\geq 1}$ corresponds to a unique number $\alpha $. The convergents $\left( p_{n},q_{n}\right) \in 
\mathbb{Z}\times \mathbb{Z}^{\ast }$of $\alpha $ are defined by $a_{j}$ in
the following way%

\begin{equation}
\left\{ 
\begin{array}{ccc} \label{eq.conv}
p_{n}=a_{n}p_{n-1}+p_{n-2} & for\text{ }n\geq 2,\text{ } & p_{0}=a_{0},\text{
}p_{1}=a_{0}a_{1}+1 \\ 
q_{n}=a_{n}q_{n-1}+q_{n-2} & for\text{ }n\geq 2, & q_{0}=1,\text{ }%
q_{1}=a_{1}.%
\end{array}%
\right.
\end{equation}
For an arbitrary choice $n_0$ and an arbitrary choice of $a_n^{(j)}$, for $j\in \jjj, n\leq n_0$, it is straightforward from \eqref{eq.conv} to construct inductively $a_{n_0+1}^{(1)},a_{n_0+1}^{(2)}, a_{n_0+1}^{(3)},a_{n_0+1}^{(4)},a_{n_0+2}^{(1)},\ldots$ such that \eqref{Arith1} holds. Indeed, suppose all coefficients are chosen up to $a_n^{(4)}$. Then we just have to take $a_{n+1}^{(1)}>e^{nq_n^{(4)}}$, then $a_{n+1}^{(2)}>e^{(n+1)q_{n+1}^{(1)}}$, then $a_{n+1}^{(3)}>e^{(n+1)q_{n+1}^{(2)}}$, then 
$a_{n+1}^{(4)}>e^{(n+1)q_{n+1}^{(3)}}$.
Since each time we have to pick a coefficient we can choose it in an infinite semi-interval of the integers, the set $Y$ is clearly a continuum. The fact that the coefficients can be chosen arbitrarily for $n\leq n_0$ with $n_0$ arbitrarily large, implies the density of the set $Y$. 
\end{proof} Define for $j\in \jjj$
\begin{equation} \label{eq.Tnj} T_{n}^{j}=[e^{nq_{n}^{(j)}},\frac{q_{n+1}^{(j)}}{n+1}]\end{equation}

The important consequence from our definition of the set of $Y$ is the following 
\begin{lemma} \label{lemma.Tnj} Every  $t \geq 0$ belongs to at least three intervals $T_{n_j}^{j}$. More precisely one of the following holds 
\begin{enumerate}
\item $t\in T_n^1 \bigcap T_n^2 \bigcap T_n^3$
\item $t\in T_n^2 \bigcap T_n^3 \bigcap T_n^4$
\item $t\in T_{n}^3 \bigcap T_n^4 \bigcap T_{n+1}^1$
\item $t\in T_n^4 \bigcap T_{n+1}^1 \bigcap T_{n+1}^2$
\end{enumerate}
\end{lemma}

\begin{proof}[Proof of Lemma \ref{lemma.Tnj}]
Let $n$ be such that $t\in [e^{nq_n^{(1)}},e^{(n+1)q_{n+1}^{(1)}}]$. If 
$t\leq q_n^{(2)}/n$, then $t\in T_{n-1}^j$ for every $j=2,3,4$. If $t\in [q_n^{(2)}/n,q_n^{(3)}/n]$, then $t\in T_{n-1}^j$, for $j=3,4$ and $t\in T_n^{(1)}$.   If $t\in [q_n^{(3)}/n,q_n^{(4)}/n]$, then $t\in T_{n-1}^4$,  $t\in T_n^{(j)}$, for $j=1,2$.  
\end{proof}

\subsection{The choice of the ceiling function} \label{sec.ceiling} Following \cite{Y,Fanalytic}, let $\varphi$ be the following strictly positive real analytic function on $\mathbb{T}^{4}$
\begin{equation*}\label{Ceiling function}
 \varphi(x_{1},...,x_{4})=1+\sum_{j=1}^4\sum_{n\geq n_0} \frac{\cos(2\pi q^{(j)}_{n}x_{j})}{e^{q^{(j)}_{n}}},
\end{equation*}
where $n_0$ is chosen sufficiently large so that $\vphi$ is strictly positive. 

\subsection{The special flows}\label{sec.special} Now, for  $\a \in Y$ and the function $\varphi$, we denote by $\{T^t_{\a,\vphi}\}$ the special flow above the translation $R_\a$ on $\T^4$ and under the ceiling function  $\vphi$ (see Section \ref{sec.def} for the definitions). Our main result is the following 
\begin{theorem}\label{Theo1}
For any $\alpha\in Y$, the special flow constructed over the translation $R_{\alpha}$ on $\mathbb{T}^{4}$ and under the ceiling function $\varphi$ has countable Lebesgue spectrum.
\end{theorem}

From Theorem \ref{Theo1} and the correspondence between special flows above translations and reparametrizations of translation flows on the torus, we derive the following result that gives a precise example of a flow as in Theorem \ref{Theo}.

\begin{corollary}\label{main}
 For any $\alpha \in Y$, there exists a strictly positive real analytic function $\Phi$ defined on $\mathbb{T}^{5}$, such that the reparametrization of the irrational flow $(\alpha,1)$ by $1/\Phi$ has countable Lebesgue spectrum. 
\end{corollary}

\begin{proof} We sketch the proof and refer to \cite[Proposition 6]{Fanalytic} for the details. Define 
$$\Phi(x_1,\ldots,x_5)=  1+ \Re\left(\sum_{n=n_0}^{\infty} d^{(j)}_n e^{i2\pi (q_{n}^{(j)}x_{j}+l_n^{(j)} x_{5})}\right)$$ 
where we choose $l_{n}^{(j)}$ to be the closest relative integer to $-q_{n}^{(j)}\alpha_j$, and $$d^{(j)}_n=\frac{i2\pi (q_n^{(j)}\a_j+l_n^{(j)})}{e^{i2\pi (q_n^{(j)}\a_j+l_n^{(j)})}-1}e^{-q_n^{(j)}}.$$
Then, if $n_0$ is sufficiently large, $\Phi$ is a real analytic strictly positive function on $\T^5$, that satisfies 
$$\vphi(x_1,\ldots,x_4)=\int_0^1 \Phi(x_1+s\a_1,\ldots,x_4+s\a_4,s)ds.$$
Hence the reparametrization of the flow of frequency $(\a,1)$ with the function $1/\Phi$ can be viewed as the special flow above $R_\a$ on $\T^4$ with the ceiling function $\vphi$. Hence the function $\Phi$ satisfies the conclusion of Corollary \ref{main}.\end{proof}

\subsection{Square summable decay for smooth coboundaries}

We recall that $f$ is called a smooth coboundary over the flow $T^t_{\a,\vphi}$ if there exists a smooth function 
$\phi$ such that, for any $a<b$, 
$$\int_a^b f(T^t(x_0,t_0)) dt = \int_a^b f(x_0,t_0+t) dt = \phi(x_0,t_0+b)-\phi(x_0,t_0+a).$$
The function $\phi$ is called the {\it transfer function} of $f$. The space of smooth coboundaries is dense in the subspace $L^2_0(M,\mu) \subset L^2(M,\mu)$ of zero average functions, provided $T^t_{\a,\vphi}$ is ergodic (which is always the case if $\a$ is irrational). The set of smooth functions that are supported inside $M_\zeta$ for some $\zeta>0$ is also dense
in the subspace $L^2_0(M,\mu) \subset L^2(M,\mu)$ of zero average functions.

Hence to prove that the maximal spectral type of  $T^t_{\a,\vphi}$ is absolutely continuous with respect to Lebesgue measure on $\R$, it suffices to show the following 
 \begin{theorem} \label{coro.abs} For every $\zeta>0$, for every smooth coboundary function $f$ and every smooth function $g$ supported in $M_\zeta$, there exists a constant $C(\a,\vphi,\zeta)>0$ such that, for every $t$, it holds that 
 \begin{equation*}
    \left\lvert \int_{M}f(T^{t}_{\alpha,\varphi}z)g(z)d\mu\right\rvert\leq  C\None(f,g)t^{-1/2-\eps}.
\end{equation*}
\end{theorem} 
The symbol $\None(f,g)$ will denote positive constant depending on the $C^0$ and the $C^1$ norm of the functions $f$ and $g$ (see Section 5 for the definitions).

\subsection{Lebesgue spectrum with infinite multiplicity} 
To show that the spectral type is Lebesgue and has infinite multiplicity we apply a general criterion, based on correlation decay, similar to the one used in \cite{FFK}. We state a version of the criterion adapted to our context in Section \ref{sec.CILS} and then we show that it holds for the flow 
$T^t_{\a,\vphi}$ defined in Section \ref{sec.special}.

\section{Birkhoff sums and their uniform stretch on good intervals}
 
First of all we state a standard result on the uniform behavior of $N(x,t)$ due to the unique ergodicity of $R_\a$ and continuity of $\vphi$.  
\begin{lemma} \label{lemma.uniqueergodic}
For any $t$ sufficiently large, for any $x\in \mathbb{T}^{4}$
\begin{equation*}
    N(x,t)\in[\frac{t}{2},2t].
\end{equation*}
\end{lemma}
\begin{proof} By the definition of $N(x,t)$, $$0\leq t-\varphi_{N(x,t)}(x)\leq \varphi(R^{N(x,t)}_{\alpha}x)\leq \|\varphi\|.$$ This shows that $N(x,t) \to \infty$ uniformly as $t\to \infty$. Since by unique ergodicity of $R_\a$ we have that $\varphi_{N(x,t)}/N(x,t)\to \int_{\T^4} \varphi(x)dx=1$, we get the bounds of the lemma.
\end{proof}

 For $n\in \N$ and $\theta>0$ we define for every $j\in \{1,\ldots,4\}$ 

$$\cW(n,\theta,j)=\left\{x \in \mathbb{T}^{1} / \{q_{n}^{(j)}x\}\in [\theta,\frac{1}{2}-\theta]\cup[\frac{1}{2}+\theta,1-\theta]\right\}$$

\begin{definition}[Good intervals]  We say that a $j$-interval $I^j \subset M$ is $(n,\theta,j)$-good if $\pi_j(I^j) \bigcap \cW(n,\theta,j) =\emptyset$.
 \end{definition}

Recall the definition of  $T_{n}^{j}=[e^{nq_{n}^{(j)}},\frac{q_{n+1}^{(j)}}{n+1}]$ given in \eqref{eq.Tnj}. Let
$${\bar T}_{n}^{j}:=[e^{nq_{n}^{(j)}}/10,10\frac{q_{n+1}^{(j)}}{n+1}], \quad j=1,...,4.$$
 
 The following estimate on uniform stretch is derived from a similar one in \cite{Fanalytic} where it plays the key role for mixing of the reparametrized flows on $\T^3$
  
 \begin{proposition} \label{propo.us} 
If $I^j$ is a $j$-interval that is $(n,\theta,j)$-good, then for any $m\in \bar{T}^{j}_{n}$ the following holds 
\begin{enumerate}
\item  $\underset{x\in I^j}{\inf}\mid\frac{\partial\varphi_{m}}{\partial{x}}\mid\geq \frac{\theta mq_{n}^{(j)}}{e^{q_{n}^{(j)}}},$
\item $\underset{x\in I^j}{\sup}\mid\frac{\partial^{2}\varphi_{m}}{\partial{x^{2}}}\mid \leq Cm.$
\end{enumerate}
\end{proposition}
\begin{proof}

For any $m\in \mathbb{N}$, we have 
\begin{align*}
\varphi_{m}(x)=m+\text{Re}\bigr( \sum_{j=1}^{4} \sum_{n=1}^{\infty} \frac{X_{j}(m,n)}{e^{q_n^{(j)}}}e^{i2\pi q_n^{(j)} x_{j}}\bigl),
\end{align*}
where
\begin{equation*}
     X_{j}(m,n)=\frac{1-e^{i2\pi m q_n^{(j)}\alpha_{j}}}{1-e^{i2\pi q_n^{(j)}\alpha_{j}}},
\end{equation*} The proof of Proposition \ref{propo.us} is straightforward from the following lemma and we refer the reader to \cite{Fanalytic} for  the details.
\begin{lemma}[\cite{Fanalytic}] \label{Lem}
We have the following inequalities:
\begin{itemize}
    \item For all $m\in \mathbb{N}$, $\mid X_{j}(m,n)\mid \leq m.$
    \item For all $m< q_{n}^{(j)}$,$m\in \mathbb{N}$,$\mid X_{j}(m,n)\mid \leq q_{n}^{(j)}.$
    \item For any $m\leq \frac{q_{n+1}^{(j)}}{2}$, $\mid X_{j}(m,n)\mid \geq \frac{2m}{\pi}.$
\end{itemize}
\end{lemma}
\begin{proof}[Proof of Lemma \ref{Lem}.]  The proof of the first inequality is obvious. The others follow from the fact that  
\begin{equation*}
    \mid X_{j}(m,n)\mid =\left|\frac{\sin(\pi  m q_n^{(j)} \alpha_{j})}{\sin(\pi  q_n^{(j)} \alpha_{j})}\right|,
\end{equation*}
and the fact that by the definition of denominators of best approximations, $q_{n}^{(j)}$ satisfies
\begin{equation*}
    \lVert q_{n-1}^{(j)}\alpha_{j}\rVert< \lVert l\alpha_{j}\rVert, \quad \mbox{  $\forall l<q_{n}^{(j)}, l\neq q_{n-1}^{(j)}$},
\end{equation*}
as well as \begin{equation*}
   \frac{1}{q_{n}^{(j)}+q_{n+1}^{(j)}}\leq \lVert q_{n}^{(j)}\alpha_{j}\rVert<\frac{1}{q_{n+1}^{(j)}}.
\end{equation*}
\end{proof}
\end{proof}
 
 As an immediate corollary of Proposition \ref{propo.us}  and the definitions \eqref{def.rt} and \eqref{def.st} and Lemma \ref{lemma.uniqueergodic},  we get that 
\begin{corollary} \label{cor.stretch}
If $I^j$ is a $j$-interval that is $(n,\theta,j)$-good, then for any $t\in T^{j}_{n}$ the following holds 
\begin{enumerate}
\item  $r_{I^j}^t \geq \theta t^{1-\epsilon/100}$ 
\item $S_{I^j}^t \geq \theta^2 t^{1-\epsilon/50}$
\end{enumerate}
\end{corollary}

\section{From uniform stretch to absolutely continuous spectrum. Proof of Theorem \ref{coro.abs}}

\subsection{Decay of correlations on good intervals due to uniform stretch}

In all this section, we consider a pair of function $(f,g)$ such that : $f$ is a smooth coboundary with transfer function $\psi$;  $g \in {C^1}(M)$ and is supported inside $M_\zeta$. We define for every $j\in \{1,\ldots,4\}$
\begin{align}
\label{eq:Norms}
\Nzero(f,g)&:= \Vert \psi \Vert_0 \Vert  g \Vert_0  \quad \text{and} \\ \None(f,g,j)&:= 
(\Vert f \Vert_0 +\Vert \psi \Vert_0) \Vert g \Vert_{1,j} +  (\Vert f \Vert_{1,j} +\Vert \psi \Vert_{1,j}) \Vert g \Vert_0\,,
\end{align}
where $\Vert \cdot \Vert_0$ denotes the $C^0$ norm and $\Vert \psi \Vert_{1,j}=\|\psi\|_0+\|\partial_{x_j}f\|_0$. 
We also define $\None(f,g)=\max_{j\in \jjj} \None(f,g,j)$.

In \cite{Fanalytic} the uniform stretch estimates above good intervals as in Corollary \ref{cor.stretch} (with $\theta$ bounded from below) are used to prove equi-distribution of good intervals and then mixing by a Fubini argument. For the proof of countable Lebesgue spectrum, we look for square summable estimates on the decay of correlations, and this is why we use observables that are coboundaries.  We refer to \cite[Proposition 6.2]{FFK} for the proof of the following proposition that relates uniform stretching of the Birkhoff sums and the decay of correlations.

\begin{proposition} \label{propo.mix} 

For any $j$-interval $I^j\in M$ with extremities $(a,s)$ and $(b,s)$ and for any  $t\in T_n^{j}$, we have
\begin{equation*}
    \left\lvert \int_{I^j}f(T^{t}_{\alpha,\varphi}z)g(z)d\lambda(z)- \Delta(I^j)\right\rvert\leq C \left(\Nzero(f,g)\frac{\lambda(I^j)}{S^{t}_{I^j}}+\None(f,g,j)\frac{\lambda(I^j)}{r^{t}_{I^j}}\right),
\end{equation*}
where $\Delta(I^j)=\frac{g(a,s)\psi(T^{t}(a,s))}{\partial_j\varphi_{N(a,s)}(a)}-\frac{g(b,s)\psi(T^{t}(b,s))}{\partial_j\varphi_{N(b,s)}(b)}$.
\end{proposition}
\begin{remark} The only difference between Proposition \ref{propo.mix} and Proposition 6.2 of \cite{FFK}, is the use of the norm $\None(f,g,j)$ to capture the fact that the interval of integration is in the direction of $x_j$. In \cite{FFK} the special flow was considered above a circular rotation and the intervals were as a consequence all in the same direction.  It is  natural and transparent from the proof of Proposition 6.2 of \cite{FFK} that the derivative that appear in the estimate of the integral   $ \int_{I^j}f(T^{t}_{\alpha,\varphi}z)g(z)d\lambda(z)$ is in the $j$-direction. 
\end{remark} 

As a direct consequence of Propositions \ref{propo.us} and \ref{propo.mix} we have the following crucial estimate on the decay of correlations of the flow $T^{t}_{\alpha,\varphi}$ (when, as assumed in all this section, one of the functions is a smooth coboundary): 
\begin{corollary} \label{coro.good} If $I^j$ is a $j$-interval that is  $(n,t^{-1/4+\epsilon},j)$-good, and if $t\in T_n^{j}$, it holds that 
 \begin{equation*}
    \left\lvert \int_{I^j}f(T^{t}_{\alpha,\varphi}z)g(z)d\lambda(z)\right\rvert\leq C\None(f,g,j)\lambda(I^j)t^{-1/2-\epsilon}.
\end{equation*}
\end{corollary} 

\subsection{Partial partitions into good intervals} 

From Lemma \ref{lemma.Tnj}, every  $t \in \R$ belongs to at least three intervals $T_{n_j}^{j}$. 
Fix now $t$ such that $t\in T_n^1 \bigcap T_n^2 \bigcap T_n^3$, the other cases being treatable in exactly a similar fashion.
\begin{proposition}Ê\label{propo.partition} For any $\zeta>0$, if $t$ is sufficiently large, there exists a partial partition of $M$, $\cG_t$ such that the atoms of $\cG_t$ are $j$-intervals that are $(n,t^{-1/4+\epsilon},j)$-good, with $j\in \{1,2,3\}$ and $\mu(\cG_t^c\cap M_\zeta)\leq t^{-1/2-\epsilon}$. 
\end{proposition}

Note that we cannot take $j$ to be the same for all the intervals in $\cG_t$. 

\begin{proof} We first consider the partial partition $\cP_1$ of $\T$ into intervals that are the connected components of $\cW(n,t^{-1/4+\epsilon},1)=\left\{x \in \mathbb{T}^{1} / \{q_{n}^{(j)}x\}\in [\theta,\frac{1}{2}-\theta]\cup[\frac{1}{2}+\theta,1-\theta]\right\}$. Next we consider on $\T^4$ the partial partition $\cG^{(1)}$ consisting of $1$-intervals of the form $I\times (x_2,x_3,x_4)$ where $I$ ranges over all the intervals of $\cP_1$ and where $(x_2,x_3,x_4)$  ranges over all points in $\T^3$. The union of the atoms of $\cG^{(1)}$ covers all $\T^4$ except for a set $\cE_1$ that is the union of the bands $x_1 \in \Delta_{n,k}^{(1)}$ and $x_1 \in \bar{\Delta}_{n,k}^{(1)}$, where $\Delta_{n,k}^{(1)}=[\frac{k}{q_n^{(1)}}-\frac{t^{-1/4+\epsilon}}{q_n^{(1)}},\frac{k}{q_n^{(1)}}+\frac{t^{-1/4+\epsilon}}{q_n^{(1)}}]$ and $\bar{\Delta}_{n,k}^{(1)} =\frac{1}{2q_n^{(1)}} +\Delta_{n,k}$, and $k\in \{0,\ldots,q_n^{(1)}-1\}$.

Next, we consider  the partial partition $\cP_2$ of $\T$ into intervals that are the connected components of $\cW(n,t^{-1/4+\epsilon},2)$. Next, we take a partial partition $\cG^{(2)}$ of the set $\cE_1$ into $2$-intervals that are of the form $x_1 \times I \times (x_3,x_4)$ where $I$ ranges over all intervals of $\cP_2$ and $(x_3,x_4)$ over all points in $\T^2$ and $x_1\in \Delta$ where $\Delta$ ranges over  all bands $\Delta_{n,k}^{(1)}$ and $\bar{\Delta}_{n,k}^{(1)}$, for $k\in \{0,\ldots,q_n^{(1)}-1\}$.

 The union of the atoms of $\cG^{(1)}$ and $\cG^{(2)}$ covers all $\T^4$ except for a set $\cE_2$ that is the union of the bands $(x_1,x_2) \in \Delta \times \Delta'$ where $\Delta$ ranges over all the sets $\Delta_{n,k}^{(1)}$ and   $\bar{\Delta}_{n,k}^{(1)}$, for $k\in \{0,\ldots,q_n^{(1)}-1\}$ and $\Delta'$ ranges over all the sets ${\Delta}_{n,k}^{(2)}$ and   $\bar{\Delta}_{n,k}^{(2)}$, for $k\in \{0,\ldots,q_n^{(2)}-1\}$.
 
 Finally, we consider  the partial partition $\cP_3$ of $\T$ into intervals that are the connected components of $\cW(n,t^{-1/4+\epsilon},3)$. Next, we take a partial partition $\cG^{(3)}$ of the set $\cE_2$ into $3$-intervals that are of the form $(x_1,x_2) \times I \times x_4$ where $I$ ranges over all intervals of $\cP_3$ and $x_4$ over all points in $\T$ and $(x_1,x_2)$ over all points in the projection of $\cE_2$ on the first and second coordinates of $\T^4$.

We let $\cG$ be the partial partition of $\T^4$ consisting of the atoms of $\cG^{(1)}$ and $\cG^{(2)}$ and $\cG^{(3)}$. The union of the atoms of $\cG_t$ covers all $\T^4$ except for a set $\cE_3$ that is the union of the one dimensional bands $(x_1,x_2,x_3) \in \Delta \times \Delta' \times \Delta''$ where $\Delta$ ranges over all the sets $\Delta_{n,k}^{(1)}$ and  $ \bar{\Delta}_{n,k}^{(1)}$, for $k\in \{0,\ldots,q_n^{(1)}-1\}$,  and $\Delta'$ ranges over all the sets ${\Delta}_{n,k}^{(2)}$ and   $\bar{\Delta}_{n,k}^{(2)}$ for $k\in \{0,\ldots,q_n^{(2)}-1\}$,   and $\Delta''$ over all the sets  ${\Delta}_{n,k}^{(3)}$ and   $\bar{\Delta}_{n,k}^{(3)}$ for $k\in \{0,\ldots,q_n^{(3)}-1\}$. Clearly the set $\cE_3$ has a Haar measure on $\T^4$ bounded by $4^3t^{-3/4+3\epsilon}=o(t^{-1/2-\eps})$. 
 
To conclude, we want to extend the partial partition $\cG$ to the phase space $M$ of the special flow $\{T^t_{\a,\vphi}\}$. It is here that we need to use the sets $M_\zeta$. For a fixed $\zeta>0$, we define for sufficiently large $t$, the partial partition $\cG_t$ as follows. For every interval $I \in \cG$ we include in $\cG_t$ all the intervals of the form $I\times s$ that satisfy $I\times s \cap  M_{\zeta} \neq \emptyset$ and  $I\times s \cap  M_{\zeta/2}^c =\emptyset$.    
 The latter condition insures that all the intervals we end up including in $\cG_t$ are all disjoint since they were disjoint in $\cG$. The former condition insures that   $\mu(\cG_t^c\cap M_\zeta)\leq \lambda_4(\cG)=o(t^{-1/2-\eps})$. \end{proof}

\subsection{Square summable decay of correlations. Proof of Theorem \ref{coro.abs}}

Recall that by Proposition \ref{propo.partition} we have a partial partition $\cG_t$ in good intervals of various directions. Since $\mu(\cG_t^c\cap M_\zeta) \leq t^{-1/2-\eps}$, and since $g$ vanishes on $M_\zeta^c$, we can apply the corollary in the various directions depending on the interval of the partial partition that we consider, we get by Fubini the required decay. More precisely, for every atom of $\cG_t$ that is a good interval $I$ in some direction $j$, we have from Corollary \ref{coro.good} that  
 \begin{equation*}
    \left\lvert \int_{I}f(T^{t}_{\alpha,\varphi}(z))g(z)d\lambda(z)\right\rvert\leq C\None(f,g)\lambda(I)t^{-1/2-\eps}.
\end{equation*}
By Fubini, we get 
\begin{equation*}
    \left\lvert \int_{\cG_t}f(T^{t}_{\alpha,\varphi}(z))g(z)d\mu(z)\right\rvert\leq C\None(f,g)t^{-1/2-\eps},
\end{equation*}
which implies, since $g$ vanishes on $M_\zeta^c$ that 
\begin{equation*}
    \left\lvert \int_{\cG_t\cup M_\zeta^c}f(T^{t}_{\alpha,\varphi}(z))g(z)d\mu(z)\right\rvert\leq C\None(f,g)t^{-1/2-\eps},
\end{equation*}
which implies, since $\mu(\cG_t^c \cap M_\zeta)\leq t^{-1/2-\eps}$ that 
\begin{equation*}
    \left\lvert \int_{M}f(T^{t}_{\alpha,\varphi}(z))g(z)d\mu(z)\right\rvert\leq C\None(f,g)t^{-1/2-\eps},
\end{equation*}
as required. \hfill $\Box$

\section{Infinite Lebesgue Spectrum} \label{sec.CILS}
\subsection{Criterion for Infinite Lebesgue Spectrum}
 In this subsection, we state a version adapted to our context of the criterion for countable Lebesgue spectrum that was proved in \cite{FFK}. For a multi-interval $J=I_1\times \ldots \times I_4 \subset \T^4$,  let $T_J$ be the maximal real number such that $T^t(J,0) \cap (J,0)=\emptyset$ for every $|t|<T_J$. 

For $T\leq T_J$ we define the tower above $J$
$$R_J^{T}:= \bigcup_{t\in (-T,T)} F^{t}(J,0),$$ 
and define the flow-box $F_J^{T}$ above $J$ with range $R_J^T$ as
\begin{equation*}
F^T_J (x,t) = T^t (x,0)         \,, \,\, \text{ \rm for all }  (x,t) \in J \times (-T,T).
\end{equation*}
The flow-box $F_J := F^{T_J}_J$ will be called a {\it maximal flow-box} over the {\it basis} $J\subset M$.

 Given a flow-box  {$F_J^T$}, we define, for any $\zeta>0$, the set $S^T_\zeta(J) \subset \R$ 
as follows
$$
S^T_\zeta(J):=\{ t\in (-T,T) \;:\;  T^t (J) \cap M_\zeta^c =\emptyset\}\,.
$$ 
 By definition we have that $S^T_\zeta(J)$ is an open subset (which in general may be empty).

 We can now define the functions  supported on flow-boxes that we will be working with. 

\begin{definition}
\label{def.class}
Given a flow-box $F_J^T$ and constants $C,\zeta>0$, we define $\mathcal G(J,T,C,\zeta)$ to be the class  of all functions  $g\in C^\infty(M)$, that vanish outside $R^T_J$, while on $R^T_J$ they are defined as  
\begin{equation*}
g\left(F^T_J(x, t)\right):=   \chi_J(x)  \psi (t) \,,  \quad  \text{ \rm if }  (x,t) \in J \times (-T,T) \,,\\
\end{equation*} 
with $\psi \in C^\infty_0 (\R,\R)$, $\psi$ vanishes outside $S^T_\zeta(J)$, and has  $C^1$ norm bounded above by $C$, and 
$\chi_J \in C_0^\infty(J)$ such that $\int_J \chi^2_J d\lambda_J =1$,  and  
\begin{equation} \label{eq.chi} \|\chi_J\|_0\leq C\lambda_4(J)^{-1/2}, \quad \|\chi_J\|_{1,j}\leq C\lambda_4(J)^{-1/2}|I_j|^{-1}.\end{equation} 

The class $\mathcal F(J,T,C,\zeta)$ is the subset of  $\cG(J,T,C,\zeta)$ consisting of smooth coboundaries.

\end{definition} 
The general criterion for countable Lebesgue spectrum stated in \cite{FFK} implies in our context the following.
We use the notation 
$$\langle f \circ T^t , g \rangle=\int_{M}  f\circ T^t_{\alpha, \vphi} (z) g(z) d\mu(z) \,.$$

\begin{theorem} \label{theo.criterion} Assume $\{T^t_{\a,\vphi}\}$ has an absolutely continuous maximal spectral type.  If there exists a sequence of multi-intervals $J_n$ such that $\lim \lambda (J_n)=0$ and if for any $T>0$, $C>0$ and $\zeta >0$, for any family $\{(f_n, g_n)\}$  of pair of functions such that $f_n \in \mathcal F(J_n,T,C,\zeta)$ and $g_n \in \mathcal G(J_n,T,C,\zeta)$, we have 
\begin{equation}Ê\label{cond.decay} \inf_n   \int_{\R\setminus [-T_{J_n}, T_{J_n}]} \vert \langle f_n \circ T^t , g_n \rangle\vert^2  dt =0\,.
\end{equation} 
Then the flow $\{T^t_{\a,\vphi}\}$ has countable Lebesgue spectrum. 
\end{theorem} 

The statement of Theorem \ref{theo.criterion} is almost identical to Theorem 6 of \cite{FFK}, with this important difference that in the class of functions that we consider 
is more general in that we replaced in \eqref{eq.chi} the condition 
$$ \|\chi_J\|_{1}\leq C \lambda_4(J)^{-1/2-1}$$
that was used in \cite{FFK} by a less stringent and more precise one that distinguishes between $C^1$ norms according to the direction  in which the derivatives are considered. Our condition is the natural one to control the derivatives of a function supported inside $J$ and with $L^2$ norm bounded away from $0$ and $\infty$. It is important for us to have a differentiated control of the derivatives of the observables along different directions because when we will show that the decay condition \eqref{cond.decay} holds for the flow $\{T^t_{\a,\vphi}\}$, we will use mixing estimates for intervals in various directions, and these estimates naturally involve derivatives along the direction of the interval as stated in Proposition 
\ref{propo.mix}. 

In \cite{FFK}, the criterion for CILS was stated for general flow-boxes above multi-intervals,  but it was used for flows for which the decay of correlations  did not distinguish between various directions, a more precise statement such as the one given here was not necessary.

\begin{remark} In \cite{FFK}, $M_\zeta$ consisted of points in $M$ that avoid a neighborhood of the ceiling function, but that also avoid the singularity of the flow. For the latter reason if one wanted to have the range of the flow-box above an interval $J$ to spend most of its time in $M_\zeta$, then it was necessary, in addition to taking $\zeta$ small, to start with an interval $J$ whose first iterates  avoid the neighborhood of the singularity. In our case, there is no singularity and we do not need any extra assumption on the multi-interval $J$ besides its measure going to $0$. \end{remark} 

\begin{remark} Compared to \cite{FFK}, we dropped the unnecessary condition on higher derivatives since only the $C^1$ norms of the observables play a role in the decay of correlations estimates.  
\end{remark}

The proof of Theorem \ref{theo.criterion} will be given in Section \ref{secCILS}.  As in \cite{FFK}, Êthe proof will be  based on an abstract criterion that is very much the same as in Theorem 5 in \cite{FFK}. For completeness, we do include a proof of the abstract criterion, that is slightly simpler than the one given in \cite{FFK}. For the rest of the proof of Theorem \ref{theo.criterion}, we follow exactly the same steps as in the proof of Theorem 5 in \cite{FFK}, and we only insist on the differences that are imposed by  the differentiated control of the derivatives of the observables along different directions.

\subsection{Verification of the criterion for $\{T^t_{\a,\vphi}\}$} \label{sec6.kochergin}

We prove below that the hypotheses of Theorem~\ref{theo.criterion}  are verified for $\{T_{\alpha, \vphi}^t\}$.  First of all, we know from Theorem \ref{coro.abs} that $\{T^t_{\a,\vphi}\}$ has an absolutely continuous maximal spectral type.

We consider the following family of maximal flow boxes. For $j\in \jjj$, let $J_{j,n}=[1/(8q_n^{(j)}),1/(4q_n^{(j)})]$ and take $J_n=J_{1,n} \times \ldots \times J_{4,n}$. 
Note that  by Lemma \ref{lemma.uniqueergodic}, it holds that $T_{J_n}\geq q_n^{(1)}q_n^{(2)}q_n^{(3)}q_n^{(4)}$.



\begin{theorem} \label{theo.cils}  For any $T>0$, $C>0$, and $\zeta>0$,  for any sequence of pair of functions 
$\{(f_n, g_n)\}$ such that $f_n \in \mathcal F(J_n,T,C)$ and $g_n \in \mathcal G(J_n,T,C,\zeta)$ we have 
$$
\lim_{n\to 0}  \int_{\R\setminus [-T_{J_n}, T_{J_n}]} \vert \langle f_n \circ T_{\alpha, \vphi}^t , g_n \rangle\vert^2  dt =0\,.
$$
\end{theorem}

\noindent   


\begin{proof}  

Having fixed $T>0$, we have that  $T \in (0,T_{J_n}^{1/2})$ for $n$ sufficiently large. Let $\vert t  \vert \geq T_{J_n}$. 
 WLOG we can assume $t>0$ since the argument for $t<0$ is similar. 
 
Observe that the function $g_n$ is supported in $F^T_{J_n}\cap M_\zeta$. 
Next, observe that there exists a set $A_n$ that is a disjoint union of sets of the form $(R^k_\a J_n,s) \subset M, |k|\leq 10T$ (by Lemma \ref{lemma.uniqueergodic}) such that $$F^T_{J_n}\cap M_\zeta \subset A_n, \quad \mu(A_n)\leq 2\mu(F^T_{J_n}).$$

Since $g_n$ vanishes on $A_n^c$,  it is enough to prove bounds on 
$$
\int_{A_n}  f_n\circ T^t_{\alpha, \vphi} (z) g_n(z) d\mu(z) \,.
$$



First, we consider the case of $t\in [T_J/2,q^{(1)}_{n+1}/(n+1)]\subset T_n^1$. We want to use the estimate in Proposition \ref{propo.mix}. 

\begin{claim} \label{claim.decomposition} The set $A_n$ is a disjoint union of $1$-intervals that are $(n,1/50,1)$-good. 
\end{claim} 

\begin{proof} Follows directly from the fact that $A_n$ is a disjoint union of sets of the form $(R^k_\a J_n,s) \subset M, |k|\leq 10T$. 
\end{proof} 

Consider any $1$-interval $I$ in the decomposition of $A_n$ given by Claim \ref{claim.decomposition}. 
Observe that since $f_n,g_n \in \mathcal G(J_n,T,C)$ (see Definition \ref{def.class}) then 
$$\Nzero(f_n,g_n) \leq C|J_n|^{-1}, \quad \None(f,g,1) \leq Cq_n^{(1)} |J_n|^{-1}\leq t^{\epsilon/100}\lambda_4(J)^{-1}.$$
Since $f_n\in \mathcal F(J_n,T,C)$, Propositions \ref{propo.us} and \ref{propo.mix} and the fact that the $1$-interval $I$ is $(n,1/50,1)$-good, imply  
\begin{align*}
    \left\lvert \int_{I}f(T^{t}_{\alpha,\varphi}z)g(z)d\lambda(z) \right\rvert&\leq C \{\Nzero(f,g))\frac{\lambda(I)}{S^{t}_{I}}+\None(f,g,1)\frac{\lambda(I)}{r^{t}_{I}}\},\\ &\leq 
    Ct^{-1+2\epsilon}|J_n|^{-1}\lambda(I).
\end{align*}
Integrating over $I$, we get by Fubini 
\begin{align} \label{eq111}
    \left\lvert \int_{M}f(T^{t}_{\alpha,\varphi}z)g(z)d\mu(z) \right\rvert=\left\lvert \int_{A_n}f(T^{t}_{\alpha,\varphi}z)g(z)d\mu(z) \right\rvert&\leq \nonumber Ct^{-1+2\epsilon}|J_n|^{-1}\mu(A_n)\\&\leq Ct^{-1+2\epsilon}.
    \end{align}

For $t>q^{(1)}_{n+1}/(n+1)$, we use Corollary \ref{coro.abs} and the fact that $\None(f_n,g_n)\leq {(q^{(1)}_{n+1})}^{\epsilon/100}$ to see that 
\begin{equation} \label{eq222}
    \left\lvert \int_{M}f(T^{t}_{\alpha,\varphi}z)g(z)d\mu\right\rvert\leq t^{-1/2-\eps/2}.
\end{equation}
In conclusion, we have proved that  for $|t|\geq T_{J_n}$
\begin{equation*}
  \left\lvert \int_{M}f(T^{t}_{\alpha,\varphi}z)g(z)d\mu\right\rvert\leq C t^{-1/2-\eps/2}.\end{equation*}
Squaring and integrating and using $T_{J_n}\to \infty$ as $n\to \infty$ we get the desired decay. 
\end{proof}

\subsection{Proof of Theorem \ref{Theo1}}
The proof follows directly from Theorems \ref{coro.abs}, \ref{theo.criterion}, and \ref{theo.cils}. \hfill $\Box$

\section{The decay of correlations criterion for Infinite Lebesgue Spectrum}
 \label{secCILS}

\subsection{The general abstract criterion} Our criterion for countable Lebesgue spectrum of smooth flows is based on the following abstract criterion that gives lower bounds on the multiplicity of a  strongly continuous one-parameter unitary group $\{\Phi_t\}_{t\in \R}$ on 
 on a separable Hilbert space $H$ with absolutely 
continuous spectrum. The criterion is essentially the same as in Theorem 5 in \cite{FFK}. The proof  that we give of this theorem is slightly simpler than the one in \cite{FFK}. 

Before we state the criterion, recall that the spectral theorem asserts that for every $f \in H$, there exists a positive spectral measure $\nu_f$ on $\R$ such that 
$\langle f\circ \Phi_t,  f\rangle=\int_\R e(t\theta)d\nu_f(\theta)$.  
Since we assume that $\{\Phi_t\}$ has absolutely continuous spectrum, we have that $\nu_f$ is given by a positive density function with respect to Lebesgue measure.

Given a measurable set $C\subset \R$ and an $L^2$ function $G$ on $\R$, we denote by $\|G\|_{L^2(C)}$ the $L^2$ norm of the restriction of $G$ on $C$. Unless specified, $L^2$ norms will be considered with respect to the Lebesgue measure $\lambda$ on $\R$.

\begin{theorem} \label{thm:Criterion}  For a fixed $n\in \N$, let us assume that  for every bounded set $C\subset \R\setminus \{0\}$ of positive Lebesgue measure there exists $\epsilon_{n,C}>0$ such that 
the following holds.  For every $\epsilon \in (0, \epsilon_{n,C})$ 
there exist vectors $f_1, \dots, f_{n} \in H$  such that
\begin{align} \label{cils1}
\Vert \langle f_i\circ \Phi_t,  f_j\rangle  \Vert_{L^2} &\leq  \delta_{ij} +  \epsilon\,, \quad 
\text{ for all } i, j \in {1, \dots, n}\,; \\ \label{cils2}
\|F_i^2-1\|_{L^2(C)}&\leq \eps,  \quad 
\text{ for all } i \in {1, \dots, n}\,,
\end{align}
where $F_i^2(\cdot)$ is the density of the  spectral measure associated to $f_i$.

Then the spectral type of $\{\Phi_t\}_{t\in \R}$ is Lebesgue with multiplicity at least $n$.
\end{theorem}

\begin{proof} Since the maximal spectral type is absolutely continuous with respect to Lebesgue, we can assume that the sequence $(\mu_k)$ is given by $\mu_{k}=\phi_{k}(\theta)\frac{1}{1+\theta^2}d\theta$ where $\phi_{k}$ are the characteristic functions of nested measurable sets $C_k$ on $\mathbb{R}$.

Let $\bigoplus_{k\in \mathbb{N}}H_{k}$ denote the orthogonal  decomposition of $H$ into cyclic sub-spaces of $\Phi_t$ such that for all $i\in\mathbb{N}$, we have $H_{k}\simeq L^{2}(\mathbb{R},\mu_{k})$.

Let us assume by contradiction that the spectrum is not Lebesgue with multiplicity at least $n$. Then, there exists a compact set $\widehat{C}$ of positive Lebesgue measure such that $\phi_{k}=0$ on $\widehat{C}$ for every $k\geq n$. 

Take $C=\widehat{C}$, and  $f_{1},..,f_{n}\in H$ be as in Theorem \ref{thm:Criterion} with $\epsilon \ll1$ to be specified later.

Now for any $U\in L^{\infty}(\mathbb{R},\lambda)$ such that $\lVert U\rVert_{\infty}\leq 1$, take $h=h_{C,U} \in L^2(\mathbb{R})$ to be the inverse Fourier transform of $\chi_C U$: \begin{equation*}
   \chi_{C}(\theta)U(\theta)=\int_{\mathbb{R}}h(s)e(\theta s)ds.
\end{equation*}
We have that $\|h\|_{L^2}\leq \lambda(C)$.
 
 For every $i\in \{1,\ldots,n\}$, let $f_{i}^{1},f_{i}^{2},\ldots$ denote the successive orthogonal projections of $f_{i}$ on the spectral decomposition $\bigoplus L^{2}(\mathbb{R},\mu_{k})$. By the definition of the spectral isomorphism the projections of $f_i\circ \Phi_s$ are given by $e(s\cdot)f_{i}^{1}(\cdot),e(s\cdot)f_{i}^{2}(\cdot),\ldots$

 For every $i\in \{1,\ldots,n\}$, introduce the function $\nu_i\in H$:$$\nu_{i}=\int_{\mathbb{R}}h(s)f_{i}\circ \Phi_{s}ds.$$
 Since $\mu_k$ vanishes on $C$ for every $k\geq n$, we have that $\nu_i$ projects on $\chi_{C}U f_{i}^{1}\oplus \chi_{C}U f_{i}^{2}\oplus \ldots \oplus\chi_{C}U f_{i}^{n-1}$.

 We apply assumption \eqref{cils1} to $(f_{i},f_{j})$ for any $i\neq j\in \{1,\ldots,n\}^2$. By the Cauchy Schwartz inequality, we obtain
\begin{align*}
     <\nu_{i},f_{j}>&= \int_{\mathbb{R}}<h(s)f_{i}\circ \Phi^{s},f_{j}> ds  \\ &=   \int_{\mathbb{R}} h(s) <f_{i}\circ \Phi^{s},f_{j}> ds  \\ &\leq \|h\|_{L^2} \|<f_{i}\circ \Phi^{s},f_{j}>\|_{L^2} \\ &\leq \lambda(C) \eps.
\end{align*}

We compute the same quantity using the spectral identification 
\begin{equation*}
\begin{split}
    \mid <\nu_{i},f_{j}>\mid &=\mid \int_{\mathbb{R}}<h(s)f_{i}\circ \Phi^{s},f_{j}> ds \mid \\ &= 
    \mid \int_{\R}\sum_{k=1}^{\infty} \chi_C(\theta)U(\theta) f^{k}_{i}(\theta)f^{k}_{j}(\theta) \phi_{k}(\theta)\frac{1}{1+\theta^2} d\theta\mid
    \\ &=\mid \int_{C}U(\theta)(\sum_{k=1}^{n-1}f^{k}_{i}(\theta)f^{k}_{j}(\theta)) \phi_k(\theta)\frac{1}{1+\theta^2}d\theta\mid.
    \end{split}
\end{equation*}
Define
$$C_{i,j}=\{\theta \in  C: \sum_{k=1}^{n-1}f^{k}_{i}(\theta)f^{k}_{j}(\theta) > \epsilon^{1/10}\}.$$
Taking $U=\chi_{C_{i,j}}$ we get for $\eps>0$ sufficiently small that  $\lambda(C_{i,j})\leq \lambda(C)\eps^{0.85}$. Doing the same for negative values we get that if $\eps>0$ is sufficiently small \begin{equation}
 \lambda\left(\left\{\theta \in  C: \exists i\neq j\in \{1,\ldots,n\}^2, \  |\sum_{k=1}^{n-1}f^{k}_{i}(\theta)f^{k}_{j}(\theta)|>\epsilon^{1/10}\right\}\right)<\epsilon^{0.8}. \label{HH1}
\end{equation}

We now observe that 
$$F_i^2(\theta)=\sum_{k=1}^\infty |f^{k}_{i}(\theta)|^2\phi_{k}(\theta)\frac{1}{1+\theta^2}.$$
Hence, the second assumption \eqref{cils2} of the theorem implies that 
$$\| \sum_{k=1}^{n-1} |f^{k}_{i}(\theta)|^2\phi_{k}(\theta)\frac{1}{1+\theta^2}-1\|_{L^2}\leq \eps.$$
This implies that 
\begin{equation} \label{HH2}
\lambdaÊ\left(   \left\{\theta \in  C:  \exists i \in \{1,\ldots,n\}, \ \sum_{k=1}^{n-1}|f^{k}_{i}(\theta)|^2\notin [1/2,2]\right\}\right)<\epsilon^{0.9}.
\end{equation}

Finally, conditions (\ref{HH1}) and (\ref{HH2}) imply that there exists $\theta_{0}$ such that the vectors 
$$v_i=(f^{1}_{1}(\theta_{0}),..,f^{n-1}_{i}(\theta_{0})), \quad i\in \{1,\ldots,n\},$$ 
satisfy for all $(i,j)\in \{1,\ldots,n\}^2$, $j\neq i$
$$|(v_i|v_j)|<\epsilon^{0.1}, \quad (v_i|v_i)\in[\frac{1}{2},2],$$
where $(\cdot|\cdot)$ denotes the Euclidean scalar product on $\mathbb C^{n-1}$. 

Since this is impossible, we conclude that the spectrum is  Lebesgue with multiplicity at least $n$.\end{proof}

\subsection{The proof of Theorem \ref{theo.criterion}}

We want to use the assumption of Theorem \ref{theo.criterion} to check the validity of the criterion of Theorem \ref{thm:Criterion}. In this section, the proofs will be exactly the same as the ones given in Section 7 of \cite{FFK}, except for \eqref{chij} where the control on the derivatives of the compactly supported functions that we use in the construction are given in the various directions of the base $J$. Notice that the control that we require in \eqref{chij} is exactly the one that allows to guarantee the fast decay of correlations  that we obtained in Theorem \ref{theo.cils}.

First of all, we recall the following corollary of  Theorem \ref{thm:Criterion} that was proved in \cite{FFK}. 

\begin{corollary} \label{cor:criterion} Let us assume that for every $n \in \N$, for any even functions $\omega_1, \dots, \omega_{n} \in \mathcal S(\R)$ (the Schwartz space), and for any any~$\epsilon >0$,  there exists  $f_1, \dots, f_{n} \in H$ such that, for all $i, j \in \{1, \dots, n\}$, we have
$$
\Vert \langle f_i\circ \phi_t,  f_j \rangle  -  \frac{d^2}{dt^2} \omega_i \ast \omega_i(t) \delta_{ij}  \Vert_{L^2(\R)} \leq \epsilon\,.
$$
Then the spectral type of the strongly continuous one-parameter unitary group $\phi_\R$ is Lebesgue with countable multiplicity.
\end{corollary} 

The fact that Theorem \ref{thm:Criterion} Êimplies Corollary \ref{cor:criterion} is clear since the assumption of the corollary allows to produce $f_1, \dots, f_{n} \in H$ 
that satisfy \eqref{cils1} and \eqref{cils2}.

We will derive Theorem \ref{theo.criterion} from Corollary \ref{cor:criterion}. Since we only control the decay of correlations for functions  in the classes $\mathcal F(J,T,C,\zeta)$ and $\mathcal G(J,T,C,\zeta)$, we need a simple approximation lemma to approximate the target functions $\omega_1, \dots, \omega_{n}$ by  functions supported inside sets of the type $S^T_\zeta(J)$, that we take as is from \cite{FFK}.

\begin{lemma}[{\cite[Lemma 7.4]{FFK}}] \label{lemma.psi} Let $\{F_J\}$ be a family of maximal flow-boxes. For every $\epsilon>0$,  
and even function $\omega \in \mathcal S(\R)$, there exist $\tau:=\tau(\epsilon,\omega)>0$ with the following property. For every $T\geq \tau$, there exist constants $C:=C(\epsilon,\omega,T)>0$ and $\zeta:= \zeta(\epsilon, \omega, T)>0$ such that, for any $F_J\in \Phi$
with $T_J>T$, there exists an even function $\psi \in C^\infty_0 (-T,T) $ satisfying
\begin{enumerate}
\item[$(a)$] $\frac{d\psi}{dt} \in  C^\infty_0 ( S^T_\zeta(J)\cap (- S^T_\zeta(J)))$;
\item[$(b)$] the $C^{2}$ norm is bounded  above by C; 
\item[$(c)$]
\begin{equation*}  
\Vert \frac{d^2}{dt^2} (\psi \ast \psi) - \frac{d^2}{dt^2} (\omega \ast \omega)  \Vert_{L^2(\R)} <  \epsilon\,. 
\end{equation*} 
\end{enumerate}
\end{lemma}

\begin{proof}[Proof of Theorem \ref{theo.criterion}]

Let us fix $\eps>0$ and any given number $n\in \N \setminus\{0\}$ of even Schwartz functions $\omega_1, \dots, \omega_{n} \in \mathcal S(\R)$.   
Let $\Phi=\{F_J\}$ be a family of maximal flow-boxes. 

By Lemma \ref{lemma.psi}  there exists $\tau>0$ and, for all $T>\tau$,  there exists $\zeta>0$ (small) such that, for every multi-interval $J=I_1\times \ldots \times I_4 \subset \T^4$, with $T_J>\tau$, there exist even functions  $\psi_i \in  C^\infty_0 ( (-T,T)), i=1,\ldots,n,$ with the property that  $\frac{d\psi_i}{dt}  \in  C^\infty_0 ( S^T_\zeta(J))$ and with $C^{2}$ norm uniformly bounded  above by a constant $C':=C'(\epsilon,\omega_1,\ldots,\omega_{n},T)>0$, such that 
\begin{equation}  
\label{eq:conv_approx_3}
\Vert \frac{d^2}{dt^2} (\psi_i\ast \psi_i) - \frac{d^2}{dt^2} ( \omega_i \ast  \omega_i)  \Vert_{L^2(\R)} <  \epsilon/2\,. 
\end{equation} 
Let  now $\chi^{(1)} _J, \dots, \chi^{(n)} _J \in C^\infty(J)$ be functions  such that 
$$\int_J   \chi^{(i)}_J  \chi^{(j)}_J   d\lambda  =   \delta_{ij}  \,,  \quad  \text{for all } i,j \in \{1, \dots, n\}\,,$$
and such that for $j\in \{1,\ldots,4\}$
\begin{equation}Ê\label{chij} \|\chi^{(i)}_J\|_{0} \leq C''\lambda_4 (J)^{-1/2}, \quad \|\chi^{(i)}_J\|_{1,j} \leq C''\lambda_4 (J)^{-1/2}{|I_j|}^{-1},\end{equation} 
  (this is possible provided that the constant $C''$ is taken to be larger than some constant that only depends on $n$). 

Let $C>\max \{C', C''\}$.
For every $i\in \{1, \dots, n\}$,  let $f^{(i)}_J \in \mathcal F(J,T,C,\zeta)$  be the function defined
on the range $R^T_J$ of the flow-box map $F^T_J$ as 
\begin{equation*}
f^{(i)}_J \circ   F^T_J (x, t) :=   \chi^{(i)}_J (x) \frac{d}{dt} \psi_i (t) \,, \quad 
\text{ \rm if } \,\, (x,t) \in J \times (-T, T)\,,
\end{equation*} 
and defined as  $f^{(i)}_J=0$ on $M\setminus R^T_J$. 

\smallskip
We then compute the correlations. Let $T_J /2 > \max \{T, \tau/2\}$. For all $t \in [-T_J,T_J]$ we have (since the functions $\psi_1, \dots, \psi_{n+1}$ are all even)
\begin{align*}
\langle f^{(i)}_J\circ T^t, f^{(j)}_J \rangle  &=\int_J \int_{-T}^T   \chi^{(i)}_J(x) \chi^{(j)}_J(x)  
\frac{d\psi_i}{dt} (\sigma+t)   \frac{d\psi_j}{dt}(\sigma) d\sigma  dx
\\ &=  (\frac{d\psi_i}{dt} \ast   \frac{d\psi_j}{dt}) (t)   \, \delta_{ij}  = \frac{d}{dt^2} (\psi_i\ast \psi_j)(t) \delta_{ij}  \,.\end{align*}
By the assumption of Theorem \ref{theo.criterion} that we are proving,  if $\lambda_4(J)$ is small enough,  for every $i,j\in \{1,\ldots,n\}$ we have:
\begin{equation*} 
 \Vert \langle f^{(i)}_J\circ T^t, f^{(j)}_J \rangle  \Vert_{L^2(\R\setminus [-T_J, T_J])} \leq \epsilon/2\,.\end{equation*}
Note that, since the functions $\psi_i$ are supported in $[-T, T]$ and $T<T_J/2$, we also have
\begin{equation} \label{eq.outside2} \frac{d}{dt^2} (\psi_i\ast \psi_j)(t) \delta_{ij}  =0, \text{ for } t\in \R\setminus [-T_J, T_J]. 
\end{equation}
By putting together formulas~ \eqref{eq:conv_approx_3}--\eqref{eq.outside2}, 
it follows that  if $\lambda_4(J)$ is small enough (hence $T_J$ is large enough), the functions $f^{(i)}_J$, with $i \in \{1, \dots, n\}$, satisfy the assumptions of Corollary \ref{cor:criterion}:
$$
\Vert \langle f^{(i)}_J\circ T^t, f^{(j)}_J \rangle  -\frac{d^2}{dt^2}(\omega_i \ast \omega_j) \delta_{ij}  \Vert_{L^2(\R)} \leq \epsilon\,.
$$
It follows then by Corollary~\ref{cor:criterion} that, under the hypotheses of Theorem \ref{theo.criterion}, the flow $\{T^t\}$ has countable Lebesgue spectrum,  hence the argument is completed. \end{proof}

\medskip 

\subsection*{Acknowledgments.} The first and second author are very grateful to Giovanni Forni, Adam Kanigowski and Jean-Paul Thouvenot for enlightening discussions on this subject.

\end{document}